\documentclass[12pt]{amsart}
\calclayout

\usepackage{amsthm,amssymb,latexsym}

\usepackage[abbrev]{amsrefs}
\usepackage{hyperref}

\usepackage{xcolor}
\usepackage{cite}

\newcommand{\norm}[1]{\left\|#1\right\|}                

\newcommand{\KE}{\mathrm{KE}}
\newcommand{\Ric}{\operatorname{Ric}}
\newcommand{\B}{\mathbb B}

\theoremstyle{plain}
\newtheorem{thm}{Theorem}[section]
\newtheorem{lem}[thm]{Lemma}
\newtheorem{prop}[thm]{Proposition}
\newtheorem{cor}[thm]{Corollary}

\newtheorem{mainthm}{Theorem}

\theoremstyle{definition}
\newtheorem{defn}[thm]{Definition}

\theoremstyle{remark}
\newtheorem{rmk}[thm]{Remark}

\numberwithin{equation}{section}
\title{Boundary Behavior of Bisectional Curvatures for Weighted Bergman Metrics}
\author{Sungmin Yoo}
\address{Department of Mathematics, Incheon National University, 
119 Academy-ro, Yeonsu-gu, Incheon 22012, Republic of Korea}
\email{sungminyoo.math@inu.ac.kr}

\begin{document}

\maketitle


\begin{abstract}
This paper investigates the asymptotic boundary behavior of the holomorphic bisectional curvature for weighted Bergman metrics. By characterizing extremal functions via $L^2$-orthogonal projections, we establish an explicit formula for the weighted bisectional curvature. 
Utilizing the squeezing function, we then obtain quantitative upper and lower bounds for the curvature on bounded pseudoconvex domains. 
Furthermore, we prove that at strongly pseudoconvex boundary points, the bisectional curvature asymptotically coincides with that of the unit ball. 
As an application, these results provide a streamlined and unified proof for the known asymptotic behavior of the bisectional curvature of the K\"ahler-Einstein metric.
\end{abstract}

\section{Introduction}

The asymptotic behavior of the Bergman metric near the boundary of pseudoconvex domains has been a central theme in complex geometry for several decades. The foundational work of Klembeck \cite{klembeck1978kahler}, based on Fefferman’s asymptotic expansion of the Bergman kernel, first established that the holomorphic sectional curvature of a strictly pseudoconvex domain approaches a negative constant near the boundary. This analytic approach, while powerful, often obscured the underlying geometric intuition.\medskip

The Pinchuk scaling method was significantly adapted by Kim and Yu \cite{kim1996boundary} to analyze the boundary behavior of Bergman curvature invariants. This scaling approach offered a more geometric alternative to the classical analytic methods, and was later extensively reviewed and popularized in the survey by Kim and Krantz \cite{kim2003contemporary}, which has since served as a key reference for the study of Bergman metric invariants.\medskip

Following the introduction of holomorphic homogeneous regular domains by Liu-Sun-Yau \cite{liu2004canonical}, Yeung \cite{yeung2009} formalized the concept of uniformly squeezing domains. These developments led to the emergence of the squeezing function as a central tool for quantifying the proximity of a domain to the unit ball. Building on these concepts, Zhang \cite{zhang2017intrinsic} demonstrated that the squeezing function provides a systematic and direct alternative to the technical complexities of classical scaling methods.\medskip

While study of the holomorphic sectional curvature of the classical Bergman metric is  well-established, it is a natural and essential step to extend these investigations to the {\it weighted} Bergman metric and the holomorphic {\it bisectional} curvature. 
The bisectional curvature, in particular, captures the intricate interactions between different complex directions, yet its explicit representation has often been less emphasized in the existing literature.
In this work, we aim to provide a unified treatment for the weighted bisectional curvature. Our approach is characterized by its methodological simplicity; instead of the traditional Lagrange multiplier method, we utilize a structural approach based on the orthogonal projections. 
By identifying the extremal functions as projections within the weighted Bergman space, we refine the previously recorded identities and establish a canonical representation.\medskip

The primary contributions of this paper are summarized in the following theorems. First, we provide the exact representation of the bisectional curvature of weighted Bergman metrics:

\begin{mainthm}[Theorem \ref{thm: polarized bergman-fuks}]
     Let $\Omega \subset \mathbb{C}^n$ be a bounded domain with an admissible weight $\mu\in L^1(\Omega)$.
     Then the bisectional curvature of the weighted Bergman metric $g_{\Omega,\mu}$ is given by
     $$
 B_{g_{\Omega,\mu}}(p;X,Y)=2-\left(1-\frac{|\langle X,Y\rangle_{ g_{\Omega,\mu}(p)}|^2}{|X|_{g_{\Omega,\mu}(p)}^2|Y|_{g_{\Omega,\mu}(p)}^2}\right)-\frac{I^1_{\Omega,\mu}(p;X)I^1_{\Omega,\mu}(p;Y)}{I^2_{\Omega,\mu}(p;X,Y)I^0_{\Omega,\mu}(p)},
$$
where $I^k_{\Omega,\mu}$ denote the minimum integrals (see Definition \ref{def: minimum integrals}).
\end{mainthm}

By leveraging this identity alongside the squeezing function $s_{\Omega}$, we establish the following quantitative estimates for the weighted Bergman metric $g^{\rm KE}_{\Omega,m}$ with the K\"ahler-Einstein weight $\mu:=e^{-(m-1)\varphi^{KE}}$ with Ricci curvature $-1$.

\begin{mainthm}[Theorem \ref{thm: quantative estimates}]
    Let $\Omega$ be a bounded pseudoconvex domain and let $p\in\Omega$. Then, the bisectional curvature of the weighted Bergman metric $g^{\rm KE}_{\Omega,m}$ satisfies the following inequality:
\begin{align*}
    &2\frac{m(n+1)+1}{m(n+1)}\left(1-s_{\Omega}(p)^{-4(mn+1)}\right)\\
    \leq&\
    B_{\Omega,m}^{KE}(p;X,Y)+\frac{1}{m(n+1)}\left(1+\frac{|\langle X,Y\rangle_{ g^{\rm KE}_{\Omega,m}(p)}|^2}{|X|_{g^{\rm KE}_{\Omega,m}(p)}^2|Y|_{g^{\rm KE}_{\Omega,m}(p)}^2}\right)\\
    \leq&\
    2\frac{m(n+1)+1}{m(n+1)}\left(1-s_{\Omega}(p)^{4(mn+1)}\right).
\end{align*}
\end{mainthm}

These quantitative bounds naturally lead to the following asymptotic behavior as the point $p$ approaches a strongly pseudoconvex boundary point:

\begin{mainthm}[Theorem \ref{thm: asymptotic behavior of the curvatures}]
     Let $\Omega$ be a bounded pseudoconvex domain and let $\{p_j\} \subset \Omega$ be a sequence such that $\lim\limits_{j\rightarrow\infty}s_{\Omega}(p_j)=1$.
     Then, for each fixed integer $m\ge1$, we have
    $$
    \lim\limits_{j\rightarrow\infty}\left(B_{\Omega,m}^{KE}(p_j;X,Y)+\frac{1}{m(n+1)}\left(1+\frac{|\langle X,Y\rangle_{ g^{\rm KE}_{\Omega,m}(p_j)}|^2}{|X|_{g^{\rm KE}_{\Omega,m}(p_j)}^2|Y|_{g^{\rm KE}_{\Omega,m}(p_j)}^2}\right)\right)=0
    $$
    uniformly over all nonzero $X,Y\in\mathbb C^n$.
\end{mainthm}

This establishes that the boundary of the domain possesses the same asymptotic curvature structure as the unit ball, provided the squeezing number tends to one.
As a corollary, our results provide a streamlined alternative proof for the curvature asymptotics of the Kähler-Einstein metric on strongly pseudoconvex domains (see Corollary \ref{cor: ke asymptotic}), a classic result of Cheng and Yau \cite{cheng1980existence} recently revisited by Gontard \cite{gontard2019kahler}.


\section{Preliminaries}

In this section, we provide the essential definitions and properties required for our analysis. For a more detailed exposition of the weighted Bergman framework, we refer the reader to Pasternak-Winiarski \cite{pasternak1990dependence} and also the previous work of the author \cite{yoo2025invariant}.

Let $\Omega$ be a domain in $\mathbb{C}^n$ and $\mu$ be a positive measure on $\Omega$.
Consider the following $L^2$-subspace 
$$
\mathcal{A}^2(\Omega,\mu):=\left\{u\in\mathcal{O}(\Omega)\ \Big|\ \norm{u}^2_{\Omega,\mu}:=\int_{\Omega}|u|^2 d\mu<\infty\right\}.
$$
This is called the {\it weighted Bergman space} with respect to $\mu$, and the function $e^{-\varphi}$, defined by
$$
d\mu=e^{-\varphi}d\lambda,
$$
is called the {\it weight function}, where $d\lambda$ is the standard Lebesgue measure.
We say that the weight function $e^{-\varphi}$ is {\it admissible} if for any compact subset $A\subset\Omega$, there is a constant $C_A>0$ such that for all $u\in\mathcal{A}^2(\Omega,\mu)$, it satisfies the following Cauchy type inequality:
\begin{equation*}
    \sup_{z\in A}|u(z)|\leq C_A\norm{u}_{\Omega,\mu}.
\end{equation*}
In this case, the point evaluation map is a bounded linear functional on $\mathcal{A}^2(\Omega,\mu)$.
By the Riesz representation theorem, for each point $p\in\Omega$, there exist a function $K^p_{\Omega,\mu}$ satisfying that for any $u\in\mathcal{A}^2(\Omega,\mu)$,
\begin{equation}\label{eqn: reproducing property}
   u(p)=\langle u,K^p_{\Omega,\mu}\rangle_{\Omega,\mu}=\int_{\Omega}u(z)\ \overline{K^p_{\Omega,\mu}(z)}\ d\mu(z).
\end{equation}
Define the function $K_{\Omega,\mu}:\Omega\times\Omega\longrightarrow\mathbb{C}$ by
$$
K_{\Omega,\mu}(z,w):=K^w_{\Omega,\mu}(z).
$$
This is called the (off-diagonal) {\it weighted Bergman kernel function}.
Since $\mathcal{A}^2(\Omega,\mu)$ is a separable Hilbert space in this case, for any countable orthonormal basis $\{u^{j}\}$, we have
$$
K_{\Omega,\mu}(z,w)=\sum_{j} u^{j}(z)\overline{u^{j}(w)}.
$$
Let $(z_1,\ldots,z_n)$ be the standard Euclidean coordinates for $\mathbb{C}^n$ and denote $\partial_j:=\frac{\partial}{\partial z_j}$ and $\partial_{\overline{k}}:=\frac{\partial}{\partial \overline{z_k}}$.
If the diagonal part $K_{\Omega,\mu}(z):=K_{\Omega,\mu}(z,z)>0$ at $z\in\Omega$, we can consider the hermitian symmetric matrix
$$
G_{\Omega,\mu}(z)
=\left(g_{j\overline{k}}(z)\right)_{1\leq j,k\leq n},\ \ \ \ \ \ \ \ 
g_{j\overline{k}}(z)
:=
\partial_j\partial_{\overline{k}}\log K_{\Omega,\mu}(z).
$$
If $G_{\Omega,\mu}(z)$ is positive-definite everywhere, define the metric tensor $g_{\Omega,\mu}$  by
$$
\langle X,Y\rangle_{ g_{\Omega,\mu}}:=
\sum^n_{j,k=1}g_{j\overline{k}}(z)X_j\overline{Y_k},\ \ \ \ \ \ \ \ \
|X|_{ g_{\Omega,\mu}}:=\sqrt{\langle X,X\rangle_{ g_{\Omega,\mu}}}
$$
for $X:=\sum\limits_{j=1}^nX_j\partial_j,\ Y:=\sum\limits_{j=1}^nY_j\partial_j\in T^{1,0}\Omega$.
In this case, the K\"ahler metric tensor $g_{\Omega,\mu}$ is called the {\it weighted Bergman metric} of $\Omega$.
\begin{prop}[{\cite[Proposition 2.7]{yoo2025invariant}}]\label{prop: positivity of the weighted Bergman}
If $\Omega$ is a bounded domain with an admissible weight $\mu\in L^1(\Omega)$, then the diagonal part of weighted Bergman kernel is positive everywhere, and the weighted Bergman metric is positive-definite everywhere.
\end{prop}

Recall that the coefficients of the curvature tensor of the K\"ahler metric $g$ is given by
$$
R_{i\overline{j}k\overline{l}}=-\partial_k\partial_{\overline{l}}\ g_{i\overline{j}}+\sum_{\alpha,\beta} g^{\alpha\overline{\beta}}\ \partial_kg_{i\overline{\beta}}\ \partial_{\overline{l}}g_{\alpha\overline{j}}.
$$
The {\it (holomorphic) bisectional curvature} along $X,Y$ is defined by
$$
B_{g}(X,Y)
:=
\left(\sum\limits_{i,j,k,l}R_{i\overline{j}k\overline{l}}X_i\overline{X_j}Y_k\overline{Y_l}\right)\bigg/|X|_g^2|Y|_g^2.
$$
The {\it holomorphic sectional curvature} along $X$ is given by
$$
H_{g}(X):=B_{g}(X,X).
$$
The {\it Ricci curvature} along $X$ is
$$
Ric_{g}(X):=\sum_{i=1}^nB_g(X,e_i).
$$
where $\{e_i\}$: an orthonormal basis with respect to $g$.
We say that $g$ is {\it Einstein} if it has constant Ricci curvature at every points and every directions.

Let $ g_{\Omega,\mu}$ be a positive-definite weighted Bergman metric.
The {\it Bergman canonical function} with respect to $\mu$ is defined by
$$
J_{\Omega,\mu}(z):=\frac{\det G_{\Omega,\mu}(z)}{K_{\Omega,\mu}(z)}.
$$
Unlike the classic setting, in general, this is not invariant under biholomorphisms.
The logarithm of the Bergman canonical function can be considered as the Ricci potential function of the metric $g_{\Omega,\mu_{\Omega}}$, in the sense that
$$
\partial\overline{\partial}\log J_{\Omega,\mu}=-\text{Ric}_{g_{\Omega,\mu_{\Omega}}}-g_{\Omega,\mu_{\Omega}}.
$$
In particular, if $\log J_{\Omega,\mu}$ is pluriharmonic, then $g_{\Omega,\mu_{\Omega}}$ is the K\"ahler-Einstein metric.


\section{Completeness and Invariance of the weighted Bergman metrics}

\subsection{Completeness of the weighted Bergman metrics on uniformly squeezing domains}

\begin{defn}
    Let $\Omega$ be a bounded domain in $\mathbb{C}^n$, and $p\in\Omega$. Let $\mathbb{B}^n_r$ be the ball in $\mathbb{C}^n$ with the radius $r$ centered at the origin.
    \begin{itemize}
        \item The {\it squeezing number} $s_{\Omega}(p)$ at $p$ is defined by
        $$
        s_{\Omega}(p):=\sup_f\{s_{\Omega}(p;f)\}
        $$
        where 
        $
        s_{\Omega}(p;f):=\sup\{r\ |\ \mathbb{B}^n_r\subset f(\Omega)\subset \mathbb{B}^n_1\},
        $
        and $f:\Omega\longrightarrow\mathbb{B}_1^n$ is a holomorphic injective map satisfying $f(p)=0$.\medskip
        \item The function $s_{\Omega}:\Omega\longrightarrow[0,1]$ is called the {\it squeezing function} of $\Omega$.\medskip
        \item We say that $\Omega$ is a {\it uniformly squeezing domain} (or holomorphic homogeneous regular domain) if the squeezing function has positive lower bound.
    \end{itemize}
\end{defn}
\begin{rmk}
    \begin{itemize}
        \item Deng-Guan-Zhang {\cite[Theorem 2.1]{deng2012some}} proved that for any $p\in\Omega$, there exists a holomorphic injective map $f_p:\Omega\longrightarrow\mathbb{B}_1^n$ satisfying $f_p(p)=0$ and $\mathbb{B}_{s_{\Omega}(p)}^n\subset f_p(\Omega)\subset \mathbb{B}_1^n$.\medskip
        \item For a bounded strongly pseudoconvex domain $\Omega$ with $C^2$ boundary, it is well known that the squeezing function $s_\Omega(p)$ tends to $1$ as $p$ approaches any boundary point $q \in \partial \Omega$ (see \cite[Theorem 1.3]{deng2016properties}). This asymptotic behavior plays a central role in identifying the limiting geometry of the metric.
    \end{itemize}
\end{rmk}

\begin{thm}[{\cite[Theorem 4.5]{deng2016properties}}]\label{thm: DGZ}
Let $\Omega$ be a uniformly squeezing domain.
Then the Caratheodory metric of $\Omega$ is complete.
\end{thm}
In the paper, they also gave an alternative proof of the completeness of the classic Bergman metrics of uniformly squeezing domains, which were already proved by \cite{yeung2009}.
The method is called the {\it Hahn-Lu Comparison theorem}.
One can easily check by the same proof that the result also holds for the weighed Bergman metrics:
\begin{thm}[cf. {\cite[Theorem 5.1]{ahn2016positivity}}]\label{thm: Hahn-Lu}
Let $\Omega$ be a domain with an admissible weight $\mu$.
Assume that $K_{\mu}(p)>0$ at $p\in\Omega$.
Then for any $X\in\mathbb{C}^n$,
$$
C_{\Omega}(p;X)^2\leq |X|^2_{g_{\Omega,\mu}(p)},
$$
where $C_{\Omega}(p;X)$ is the Caratheodory length:
$$
C_{\Omega}(p;X):=\sup\{|df_p(X)|\ |\ f:\Omega\longrightarrow\Delta\ \ \text{holomorphic map},\  f(p)=0\}.
$$
\end{thm}
Therefore, Theorem \ref{thm: DGZ}, \ref{thm: Hahn-Lu}, and Proposition \ref{prop: positivity of the weighted Bergman} imply the following

\begin{prop}\label{prop: completness of the weightd Bergman metric}
    Let $\Omega$ be a uniformly squeezing domain with an admissible weight $\mu\in L^1(\Omega)$.
    Then the weighted Bergman metric $g_{\Omega,\mu}$ is complete.
\end{prop}

\subsection{Invariant weighted Bergman metrics}

\begin{defn}\label{def: assignment}
Let $\mathcal{D}$ be a collection of domains in $\mathbb{C}^n$ and $\mathcal{M}$ be an assignment of an admissible weight function $\mu_{\Omega}:=\mathcal{M}(\Omega)$ to each domain $\Omega\in\mathcal{D}$.
We will call the weighted Bergman kernel $K_{\Omega,\mu_{\Omega}}$ {\it “$\mathcal{M}$-Bergman kernel"} of $\Omega$.\\
Suppose that for any domain $\Omega\in\mathcal{D}$, $g_{\Omega,\mu_{\Omega}}$ is positive-definite.
In this case, we will call the weighted Bergman metric $g_{\Omega,\mu_{\Omega}}$
{\it “$\mathcal{M}$-Bergman metric"} of $\Omega$.
\end{defn}

\begin{defn}\label{def: invariant assignment}
We say that the assignment $\mathcal{M}$ is {\it invariant} (under biholomorphisms) if for any biholomorphism $F$ of $\Omega\in\mathcal{D}$, it satisfies that
$$
\mu_{F(\Omega)}\circ F=|h_F|^{2}\mu_{\Omega}
$$
for some non-vanishing holomorphic function $h_F$ (depending on $F$) of $\Omega$.
We say that an invariant assignment $\mathcal{M}$ is {\it canonical} (of level $m\in\mathbb{N}^+$) if it satisfies that
$$
\mu_{F(\Omega)}\circ F=|\det J_{\mathbb{C}}F|^{2(m-1)}\mu_{\Omega}.
$$
\end{defn}
In \cite{yoo2025invariant}, the author proved that the following
\begin{thm}[{\cite[Theorem 3.7]{yoo2025invariant}}]\label{thm: invariance}
If an assignment $\mathcal{M}$ is invariant, the $\mathcal{M}$-Bergman metric is invariant under biholomorphisms, i.e.,
$$
g_{\Omega,\mu_{\Omega}}=F^*g_{F(\Omega),\mu_{F(\Omega)}}.
$$
If an invariant assignment $\mathcal{M}$ is canonical of level $m$, the $\mathcal{M}$-Bergman kernel satisfies the following transformation formula:
$$
K_{\Omega,\mu_{\Omega}}(z)
=
K_{F(\Omega),\mu_{F(\Omega)}}(F(z))|\det J_{\mathbb{C}}F(z)|^{2m}.
$$
In this case, we call $\tilde{g}_{\Omega,\mu_{\Omega}}:=\frac{1}{m}g_{\Omega,\mu_{\Omega}}$ the $\mathcal{M}$-“normalized" Bergman metric.
\end{thm}

Let $\mathcal{M}$ be a canonical invariant assignment $\mathcal{M}$ of level $m$.
Then, the volume form
$$
K_{\Omega,\mu}^{\frac{1}{m}}(z)d\lambda(z)
$$
is invariant under biholomorphisms by Theorem \ref{thm: invariance}.
Therefore, it is natural to consider the following normalizations:
$$
\tilde{G}_{\Omega,\mu}(z)
:=\left(\tilde{g}_{j\overline{k}}(z)\right)_{1\leq j,k\leq n},\ \ \ \ \ \ \ \ 
\tilde{g}_{j\overline{k}}(z)
:=
\partial_j\partial_{\overline{k}}\log K_{\Omega,\mu}^{\frac{1}{m}}(z).
$$
$$
\tilde{J}_{{\Omega,\mu_{\Omega}}}(z):=\frac{\det \tilde{G}_{\Omega,\mu_{\Omega}}(z)}{K_{\Omega,\mu_{\Omega}}^{1/m}(z)},
$$
Note that the normalized Bergman canonical function $\tilde{J}_{{\Omega,\mu_{\Omega}}}$ is invariant under biholomorphisms.

\subsection{Bergman metric with the K\"ahler-Einstein weight}
Let $\mathcal{D}^{\rm bp}$ be a collection of bounded pseudoconvex domains in $\mathbb{C}^n$.
By the famous work of Cheng-Yau\cite{cheng1980existence} and Mok-Yau\cite{mok1983completeness}, every $\Omega\in\mathcal{D}^{\rm bp}$ admits unique complete K\"{a}hler-Einstein metric $g^{\rm KE}_{\Omega}$ with Ricci curvature $-1$.
Define a sequence of admissible assignments $\mathcal{M}^{\rm KE}_m$ on $\mathcal{D}^{\rm bp}$ for $m\in\mathbb{N}_+$ by
$$
\mathcal{M}^{\rm KE}_m(\Omega)
:=
\mu^{\rm KE}_{\Omega,m}
:=
e^{-(m-1)\varphi^{\rm KE}_{\Omega}}
=
\frac{1}{{\det\left(g^{\rm KE}_{\Omega}\right)}^{(m-1)}},
$$
for $\Omega\in\mathcal{D}^{\rm bp}$. 
Denote the $\mathcal{M}^{\rm KE}_m$-Bergman kernel by $K^{\rm KE}_{\Omega,m}:=K_{\Omega,\mu^{\rm KE}_{\Omega,m}}$ and the $\mathcal{M}^{\rm KE}_m$-Bergman metric by $g^{\rm KE}_{\Omega,m}:=g_{\Omega,\mu^{\rm KE}_{\Omega,m}}$, and the bisectional curvature by $B^{\rm KE}_{\Omega,m}:=B_{\Omega,\mu^{\rm KE}_{\Omega,m}}$.

\begin{prop}[\cite{yoo2025invariant}]
The assignment $\mathcal{M}^{\rm KE}_m$ is invariant and canonical of level $m$.
Hence, the Bergman metric with the K\"ahler-Einstein weight $g^{\rm KE}_{\Omega,m}$ is invariant. 
\end{prop}

Consider the sequence of $\mathcal{M}^{\rm KE}_m$-\emph{normalized} Bergman kernels, metrics, and curvatures:
$$
\widetilde{K}^{\rm KE}_{\Omega,m}:=\sqrt[m]{K^{\rm KE}_{\Omega,m}},\ \ \ \ 
\widetilde{g}^{\rm KE}_{\Omega,m}
:=\frac{1}{m}g^{\rm KE}_{\Omega,m},\ \ \ {\rm\ and\ } \ \ \
\widetilde{B}^{\rm KE}_{\Omega,m}
:=mB^{\rm KE}_{\Omega,m}.
$$

\begin{thm}[cf. \cite{tian1990,yoo2025invariant}]\label{thm: tian ke seq}
Suppose that $\Omega$ is an uniformly squeezing domain.
Then, we have the following uniform convergences:
$$
    \widetilde{K}^{\rm KE}_{\Omega,m}
    \rightarrow \det\left(g^{\rm KE}_{\Omega}\right),\ \ \ 
    \widetilde{g}^{\rm KE}_{\Omega,m}
    \rightarrow g^{\rm KE}_{\Omega},\ \ \ 
    \widetilde{B}^{\rm KE}_{\Omega,m}
    \rightarrow B^{\rm KE}_{\Omega}:=B_{g^{\rm KE}_{\Omega}},
$$
as $m\rightarrow\infty$.
\end{thm}
\begin{rmk}
    Since $\mu^{\rm KE}_{\Omega,m}\in C^{\infty}(\Omega)$, Proposition \ref{prop: completness of the weightd Bergman metric} implies that the (normalized) weighted Bergman metrics $\widetilde{g}^{\rm KE}_{\Omega,m}$ are {\it complete} if $\Omega$ is an uniformly squeezing domain.
\end{rmk}


\section{Bergman's minimum integrals and Bergman-Fuks formula}
Suppose that $\Omega$ is a bounded domain with an admissible weight $\mu\in L^1(\Omega)$.

\begin{defn}\label{def: minimum integrals}
Fix a point $p\in\Omega$ and non-zero vectors $X,Y\in\mathbb{C}^n$.
\emph{Minimum integrals} of the weighted Bergman kernels are defined by
\begin{align*}
I^0_{\Omega,\mu}(p)&:=\inf\left\{\|u\|_{\Omega,\mu}^2 : u\in\mathcal{A}^2_\mu(\Omega),\ u(p)=1\right\},\\
I^1_{\Omega,\mu}(p;X)&:=\inf\left\{\|u\|_{\Omega,\mu}^2 : u\in\mathcal{A}^2_\mu(\Omega),\ u(p)=0,\ \partial_Xu(p)=1 \right\},\\
I^1_{\Omega,\mu}(p;X|Y)&:=\inf\left\{\|u\|_{\Omega,\mu}^2 : u\in\mathcal{A}^2_\mu(\Omega),\ u(p)=0,\ \partial_Yu(p)=0,\ \partial_Xu(p)=1 \right\},\\
I^1_{\Omega,\mu}(p; \partial_k|_{<k})&:=\inf\left\{\|u\|_{\Omega,\mu}^2 : u\in\mathcal{A}^2_\mu(\Omega),\ u(p)=0,\ \partial_ju(p)=0,\  \partial_ku(p)=1,\ 1\leq \forall j<k \right\},\\
I^2_{\Omega,\mu}(p;X,Y)&:=\inf\left\{\|u\|_{\Omega,\mu}^2 : u\in\mathcal{A}^2_\mu(\Omega),\ u(p)=0,\ du(p)=0,\ \partial_X\partial_Yu(p)=1 \right\},
\end{align*}
where $\partial_X$ is the directional derivative along $X$. 
The minimum integral $I_{\Omega,\mu}^{1}(p;X\mid Y)$
is defined only when $X$ and $Y$ are linearly independent
over $\mathbb{C}$.
\end{defn}

By the definitions, we obtain the following monotonicity property.

\begin{lem}\label{lem: monotone}
    If $p\in\Omega_1\subset\Omega_2$ and $\mu_1\leq\mu_2$ on $\Omega_1$, then for $j=0,1,2$,
    $$
    I^j_{\Omega_1,\mu_1}(p;\cdot)\leq I^j_{\Omega_2,\mu_2}(p;\cdot).
    $$
\end{lem}

As in the case of classic Bergman kernels, the weighted Bergman kernel, length, holomorphic curvature, and the canonical functions can be represented by minimum integrals as follows.

\begin{thm} [Bergman-Fuks formula]\label{thm: bergman-fuks}
$$
 K_{\Omega,\mu}(p)=\frac{1}{I^0_{\Omega,\mu}(p)},\ \ \ \ \ 
 |X|_{g_{\Omega,\mu}(p)}^2=\frac{I^0_{\Omega,\mu}(p)
 }{I^1_{\Omega,\mu}(p;X)},
$$
$$
 H_{g_{\Omega,\mu}}(p;X)=2-\frac{\left(I^1_{\Omega,\mu}(p;X)\right)^2}{I^2_{\Omega,\mu}(p;X,X)I^0_{\Omega,\mu}(p)},\ \ \ \ \
 J_{\Omega,\mu}(p)=\frac{\left(I^0_{\Omega,\mu}(p)\right)^{n+1}}{\prod\limits_{k=1}^nI^1_{\Omega,\mu}(p; \partial_k|_{<k} )}.
$$
\end{thm}
The following theorem provides the minimum integral representation for the weighted Bergman metric (as a hermitian inner product) and the bisectional curvature. 
By employing a direct analysis of extremal functions — rather than a formal polarization — we clarify and generalize the identities previously recorded in the literature {\cite[Theorem 2.3]{kim2003contemporary}} by including the $S_g$-tensor terms omitted in previous formulations.

\begin{thm}\label{thm: polarized bergman-fuks}
Let
$
S_{g_{\Omega,\mu}}(p;X,Y):=1-\frac{|\langle X,Y\rangle_{ g_{\Omega,\mu}(p)}|^2}{|X|_{g_{\Omega,\mu}(p)}^2|Y|_{g_{\Omega,\mu}(p)}^2}.
$
Then the bisectional curvature
$$
 B_{g_{\Omega,\mu}}(p;X,Y)=2-S_{g_{\Omega,\mu}}(p;X,Y)-\frac{I^1_{\Omega,\mu}(p;X)I^1_{\Omega,\mu}(p;Y)}{I^2_{\Omega,\mu}(p;X,Y)I^0_{\Omega,\mu}(p)},
$$
and
\[
S_{g_{\Omega,\mu}}(p;X,Y)
=
\begin{cases}
\displaystyle
\frac{I_{\Omega,\mu}^{1}(p;X)}
     {I_{\Omega,\mu}^{1}(p;X\mid Y)}
=
\frac{I_{\Omega,\mu}^{1}(p;Y)}
     {I_{\Omega,\mu}^{1}(p;Y\mid X)},
&
\text{if $X,Y$ are linearly independent},
\\[8pt]
0,
&
\text{if $X,Y$ are linearly dependent}.
\end{cases}
\]
\end{thm}
\begin{rmk}
For later uses, we define the following tensor.
$$
    T_{g_{\Omega,\mu}}(p;X,Y):=2-S_{g_{\Omega,\mu}}(p;X,Y)-B_{g_{\Omega,\mu}}(p;X,Y)=\frac{I^1_{\Omega,\mu}(p;X)I^1_{\Omega,\mu}(p;Y)}{I^2_{\Omega,\mu}(p;X,Y)I^0_{\Omega,\mu}(p)}.
$$
\end{rmk}
\begin{proof}
For the reader's convenience and to avoid redundancy, we provide a unified proof that also addresses the preceding theorem. 
In this proof, we omit the explicit weight notation for simplicity, as the structure of the argument remains invariant:
$$
K(z,w):=K_{\Omega,\mu}(z,w),\ \ \ \ K(z):=K_{\Omega,\mu}(z,z),\ \ \  \langle X,Y\rangle_{ g(p)}:=\langle X,Y\rangle_{ g_{\Omega,\mu}(p)}.
$$
Recall that the Riesz representation theorem tells us that any bounded linear functional $L:H\longrightarrow\mathbb{C}$
admits a unique vector $h\in H$ satisfying
$$
L(u)=\langle u,h\rangle
$$
for all $u\in H$.
Moreover, $h^\perp=\text{ker}(L)$ and
$$
||h||=||L||:=\sup\{|L(u)|\ :\ u\in H,\  ||u||=1\}.
$$
Consider the bounded functionals $L_0,L_j,L_X,L_{XY}$ on $H:=\mathcal{A}^2_\mu(\Omega)$, defined by
$$
L_0(u):=u(p),\ \ \ L_j(u):=\partial_ju(p),\ \ \ L_{X}(u):=\partial_{X}u(p),\ \ \ L_{XY}(u):=\partial_{X}\partial_{Y}u(p).
$$
By the reproducing property (\ref{eqn: reproducing property}) and its derivatives, one can notice that the corresponding representation vectors are
$$
h_0(z)=K(z,p),\ \ \ \ \ h_j(z)=K_{\overline{j}}(z,p):=\partial_{\overline{j}}K(z,p),
$$
$$
h_{X}(z)=K_{\overline{X}}(z,p):=\partial_{\overline{X}}K(z,p),\ \ \ \ 
h_{XY}(z)=K_{\overline{X}\overline{Y}}(z,p):=\partial_{\overline{X}}\partial_{\overline{Y}}K(z,p),
$$
where $\partial_{\overline{j}}:=\frac{\partial}{\partial \overline{w_j}}$ and $\partial_{\overline{X}}:=X_j\partial_{\overline{j}}$.\\

\item[(1)] For the minimizer of $I^0_{\Omega,\mu}(p)$, apply the reproducing property to obtain
$$
\frac{1}{I^0_{\Omega,\mu}(p)}=\frac{1}{\inf\{||u||^2\ :\  u(p)=1\}}=||L_0||^2=||h_0||^2=K(p).
$$
\medskip
\item[(2)] For the minimizer of $I^1_{\Omega,\mu}(p;X)$, consider the projection operator $P_0$ onto the subspace
$$
H_0:=h_0^\perp=\text{ker}(L_0)=\{u\in H\ :\ u(p)=0\}.
$$
Let
$$
\tilde{h}_X:=P_0(h_{X})=h_{X}-\frac{\langle h_{X},h_0\rangle}{||h_0||^2}h_0.
$$
Then we have
$$
\frac{1}{I^1_{\Omega,\mu}(p;X)}=||\tilde{h}_X||^2=||P_0h_X||^2=||h_X||^2-\frac{|\langle h_{X},h_0\rangle|^2}{||h_0||^2}.
$$
The reproducing property imply
$$
||h_0||^2=K(p),\ \ \ \langle h_{X},h_0\rangle=K_{\overline{X}}(p),\ \ \ ||h_X||^2=K_{X\overline{X}}(p)
$$
so that
$$
\frac{1}{I^1_{\Omega,\mu}(p;X)}=K_{X\overline{X}}(p)-\frac{K_X(p)K_{\overline{X}}(p)}{K(p)}=\langle X,X\rangle_{ g(p)}\frac{1}{I^0_{\Omega,\mu}(p)}.
$$
\medskip
\item[(3)] For linearly independent $X,Y$, consider the minimizer of
$I_{\Omega,\mu}^{1}(p;X\mid Y)$ and the subspace
$$
\tilde{H}_Y:=h_0^\perp\cap h_Y^\perp=\text{ker}(L_0)\cap \text{ker}(L_Y)=\{u\in H\ :\ u(p)=0,\ \partial_Yu(p)=0\}.
$$
Denote the projection operator $\tilde{P}_Y:H_0\longrightarrow\tilde{H}_Y$.
Then, we have
\begin{align*}
\frac{1}{I^1_{\Omega,\mu}(p;X|Y)}&
=||\tilde{P}_Y\circ P_0(h_X)||^2
=||\tilde{P}_Y(\tilde{h}_X)||^2
=||\tilde{h}_X||^2-\frac{|\langle \tilde{h}_X,\tilde{h}_Y\rangle|^2}{||\tilde{h}_Y||^2}\\
&=\frac{1}{I^1_{\Omega,\mu}(p;X)}-I^1_{\Omega,\mu}(p;Y)\frac{|\langle X,Y\rangle_{g(p)}|^2}{\left(I^0_{\Omega,\mu}(p)\right)^2},
\end{align*}
where the last equality follows from the reproducing property:
$$
\langle X,Y\rangle_{g(p)}
=\frac{K_{X\overline{Y}}K-K_{X}K_{\overline{Y}}}{K^2}
=I^0_{\Omega,\mu}(p)\langle \tilde{h}_X,\tilde{h}_Y\rangle.
$$
Hence, one can notice that the following symmetry.
$$
I^1_{\Omega,\mu}(p;X)I^1_{\Omega,\mu}(p;Y)\frac{|\langle X,Y\rangle|^2}{\left(I^0_{\Omega,\mu}(p)\right)^2}
=1-\frac{I^1_{\Omega,\mu}(p;X)}{I^1_{\Omega,\mu}(p;X|Y)}
=1-\frac{I^1_{\Omega,\mu}(p;Y)}{I^1_{\Omega,\mu}(p;Y|X)}.
$$
If $X,Y$ are linearly dependent, then
\[
|\langle X,Y\rangle_{g(p)}|^2
=
|X|_{g(p)}^2 |Y|_{g(p)}^2,
\qquad
S_{g_{\Omega,\mu}}(p;X,Y)=0.
\]
This proves the stated formula for $S_g$ in both cases.
\medskip
\item[(4)] For the minimizer of $I^1_{\Omega,\mu}(p; \partial_k|_{<k})$, consider the subspace
$$
\tilde{H}_{<k}:=\bigcap_{j=0}^{k-1} h_j^\perp=\bigcap_{j=0}^{k-1}\text{ker}(L_j)=\{u\in H\ :\ u(p)=0,\ \partial_ju(p)=0,\ 1\leq \forall j<k\}.
$$
Denote the projection operator $\tilde{P}_{<k}:H_0\longrightarrow\tilde{H}_{<k}$.
Then, the projection of $\tilde{h}_k:=P_0(h_k)$ is given by
$$
\tilde{P}_{<k}(\tilde{h}_k)=\tilde{h}_k-\sum_{i=1}^{k-1}(\sum_{j=1}^{k-1}\tilde{g}^{\overline{j}i}\langle \tilde{h}_k,\tilde{h}_j\rangle)\tilde{h}_i
$$
where $(\tilde{g}^{\overline{j}i})$ is the inverse matrix of $(\tilde{g}_{i\overline{j}}):=(\langle \tilde{h}_i,\tilde{h}_j \rangle)$.
\begin{align*}
&\frac{1}{I^1_{\Omega,\mu}(p; \partial_k|_{<k})}
=||\tilde{P}_{<k}\circ P_0(h_k)||^2
=||\tilde{P}_{<k}(\tilde{h}_k)||^2
=\langle\tilde{P}_{<k}(\tilde{h}_k),\tilde{h}_k\rangle\\
&=||\tilde{h}_k||^2-\sum_{i=1}^{k-1}(\sum_{j=1}^{k-1}\tilde{g}^{\overline{j}i}\langle \tilde{h}_k,\tilde{h}_j\rangle)\langle\tilde{h}_i,\tilde{h}_k\rangle
=\tilde{g}_{k\overline{k}}-\sum_{i=1}^{k-1}\sum_{j=1}^{k-1}\tilde{g}^{\overline{j}i}\tilde{g}_{k\overline{j}}\tilde{g}_{i\overline{k}}
=\frac{\det\tilde{G}_k}{\det\tilde{G}_{k-1}},
\end{align*}
where $\tilde{G}_k:=(\tilde{g}_{i\overline{j}})_{1\leq i,j\leq k}$.
Note that
\begin{align*}
\tilde{g}_{i\overline{j}}:=\langle \tilde{h}_i,\tilde{h}_j \rangle
&=\langle h_i-\frac{\langle h_{i},h_0\rangle}{||h_0||^2}h_0,h_j-\frac{\langle h_{j},h_0\rangle}{||h_0||^2}h_0\rangle=\langle h_i,h_j\rangle-\frac{\langle h_{i},h_0\rangle\overline{\langle h_{j},h_0\rangle}}{||h_0||^2}\\
&=
K\left(\frac{KK_{i\overline{j}}-K_iK_{\overline{j}}}{K^2}\right)=Kg_{i\overline{j}}.
\end{align*}
Therefore, we have
$$
\prod\limits_{k=1}^n\frac{1}{I^1_{\Omega,\mu}(p; \partial_k|_{<k} )}
=\frac{g_{1\overline{1}}(p)}{I^0_{\Omega,\mu}(p)}\prod\limits_{k=2}^n\frac{\det\tilde{G}_k}{\det\tilde{G}_{k-1}}=K_{\Omega,\mu}(p)^n\det G_{\Omega,\mu}(p)=K_{\Omega,\mu}(p)^{n+1}J_{\Omega,\mu}(p).
$$
\medskip
\item[(5)] For the minimizer of $I^2_{\Omega,\mu}(p;X,Y)$, consider the subspace
$$
\tilde{H}_{d}:=\bigcap_{j=0}^{n} h_j^\perp=\bigcap_{j=0}^{n}\text{ker}(L_j)=\{u\in H\ :\ u(p)=0,\ du(p)=0\}.
$$
Denote the projection operator $\tilde{P}_{d}:H_0\longrightarrow\tilde{H}_{d}$.
The projection of $\tilde{h}_{XY}:=P_0(h_{XY})$ is given by
$$
\tilde{P}_{d}(\tilde{h}_{XY})=\tilde{h}_{XY}-\sum_{i=1}^{n}(\sum_{j=1}^{n}\tilde{g}^{\overline{j}i}\langle \tilde{h}_{XY},\tilde{h}_j\rangle)\tilde{h}_i
$$
Then, we have
\begin{align*}
&\frac{1}{I^2_{\Omega,\mu}(p;X,Y)}
=||\tilde{P}_{d}\circ P_0(h_{XY})||^2
=||\tilde{P}_{d}(\tilde{h}_{XY})||^2
=\langle\tilde{P}_{d}(\tilde{h}_{XY}),\tilde{h}_{XY}\rangle\\
&=||\tilde{h}_{XY}||^2-\sum_{i=1}^{n}(\sum_{j=1}^{n}\tilde{g}^{\overline{j}i}\langle \tilde{h}_{XY},\tilde{h}_j\rangle)\langle\tilde{h}_i,\tilde{h}_{XY}\rangle.
\end{align*}
Note that
$$
\tilde{h}_{XY}:=P_0(h_{XY})=h_{XY}-\frac{\langle h_{XY},h_0\rangle}{||h_0||^2}h_0.
$$
The reproducing property implies that
$$
||\tilde{h}_{XY}||^2=||P_0(h_{XY})||^2=\langle P_0(h_{XY}),h_{XY}\rangle=K_{X\overline{X}Y\overline{Y}}(p)-\frac{K_{\overline{X}\overline{Y}}(p)K_{XY}(p)}{K(p)},
$$
and
\begin{align*}
\langle \tilde{h}_{XY},\tilde{h}_j\rangle
&=\langle h_{XY}-\frac{\langle h_{XY},h_0\rangle}{||h_0||^2}h_0,h_j-\frac{\langle h_{j},h_0\rangle}{||h_0||^2}h_0\rangle=\langle h_{XY},h_j\rangle-\frac{\overline{\langle h_{j},h_0\rangle}\langle h_{XY},h_0\rangle}{||h_0||^2}\\
&=K_{XY\overline{j}}-\frac{K_{XY}K_{\overline{j}}}{K}=\frac{KK_{XY\overline{j}}-K_{XY}K_{\overline{j}}}{K}.
\end{align*}
Therefore,
\begin{align*}
&\frac{1}{I^2_{\Omega,\mu}(p;X,Y)}=
K_{X\overline{X}Y\overline{Y}}-\frac{K_{XY}K_{\overline{X}\overline{Y}}}{K}-\frac{1}{K^3}\sum_{i,j}g^{\overline{j}i}\left(KK_{XY\overline{j}}-K_{XY}K_{\overline{j}}\right)\left(KK_{\overline{X}\overline{Y}i}-K_{\overline{X}\overline{Y}}K_{i}\right)\\
&=K\left\{\frac{KK_{X\overline{X}Y\overline{Y}}-K_{XY}K_{\overline{X}\overline{Y}}}{K^2}-\frac{1}{K^4}\sum_{i,j}g^{\overline{j}i}\left(KK_{XY\overline{j}}-K_{XY}K_{\overline{j}}\right)\left(KK_{\overline{X}\overline{Y}i}-K_{\overline{X}\overline{Y}}K_{i}\right)\right\}
\end{align*}
Use the following identity for the curvature tensor {\cite[275p]{kobayashi1959geometry}}:
$$
R_{i\overline{j}k\overline{l}}=g_{i\overline{j}}g_{k\overline{l}}+g_{i\overline{l}}g_{k\overline{j}}-\frac{1}{K^2}\left(KK_{i\overline{j}k\overline{l}}-K_{ik}K_{\overline{j}\overline{l}}\right)+\frac{1}{K^4}\sum_{\overline{r}s}g^{\overline{r}s}\left(KK_{ik\overline{r}}-K_{ik}K_{\overline{r}}\right)\left(KK_{\overline{j}\overline{l}s}-K_{\overline{j}\overline{l}}K_{s}\right).
$$
The bisectional curvature is given by
\begin{align*}
&B_{g_{\Omega,\mu}}(p;X,Y)
:=\frac{\sum\limits R_{i\overline{j}k\overline{l}}(p)X_i\overline{X_j}Y_k\overline{Y_l}}{\langle X,X\rangle_{g(p)}\langle Y,Y\rangle_{g(p)}}\\
&=1+\frac{\langle X,Y\rangle_{g}\langle Y,X\rangle_{g}}{\langle X,X\rangle_{g}\langle Y,Y\rangle_{g}}-\frac{1}{K^2}\frac{\left(KK_{X\overline{X}Y\overline{Y}}-K_{XY}K_{\overline{X}\overline{Y}}\right)}{\langle X,X\rangle_{g}\langle Y,Y\rangle_{g}}\\&\ \ \ +\frac{1}{K^4}\frac{\sum_{\overline{r}s}g^{\overline{r}s}\left(KK_{XY\overline{r}}-K_{XY}K_{\overline{r}}\right)\left(KK_{\overline{X}\overline{Y}s}-K_{\overline{X}\overline{Y}}K_{s}\right)}{\langle X,X\rangle_{g}\langle Y,Y\rangle_{g}}\\
&=1+I^1_{\Omega,\mu}(p;X)I^1_{\Omega,\mu}(p;Y)\left\{\frac{|\langle X,Y\rangle_{g(p)}|^2}{\left(I^0_{\Omega,\mu}(p)\right)^2}-\frac{1}{I^2_{\Omega,\mu}(p;X,Y)I^0_{\Omega,\mu}(p)}\right\},
\end{align*}
as we required.
\end{proof}

\section{The weighted Bergman kernel, metric, and curvatures of the ball}
Let $\mathbb{B}^n_r$ be the ball in $\mathbb{C}^n$ with the radius $r$ centered at the origin.
Then, the classic Bergman kernel of $\mathbb{B}^n_r$ is
$$
K_{\mathbb{B}^n_r}(z):=\frac{n!}{(\pi r^2)^n}\left(\frac{r^2-|z|^2}{r^2}\right)^{-(n+1)}.
$$
The corresponding classic Bergman metric $g_{\mathbb{B}^n_r}$ is the unique complete K\"ahler-Einstein metric with Ricci curvature $-1$.
Consider the weight function
$$
\varphi_r:=\varphi^{\rm KE}_{\mathbb{B}^n_r}=\log\det(g_{\mathbb{B}^n_r})=\log\left(\frac{n+1}{r^2}\right)^n\left(\frac{r^2-|z|^2}{r^2}\right)^{-(n+1)}
$$
and the sequence of admissible weighted measures:
$$
d\mu^{m}_{r}:=d\mu^{\rm KE}_{\mathbb{B}^n_r,m}=e^{-(m-1)\varphi_r} d\lambda=\left(\frac{n+1}{r^2}\right)^{-(m-1)n}\left(\frac{r^2-|z|^2}{r^2}\right)^{(m-1)(n+1)}d\lambda.
$$

\begin{prop}\label{prop: bergman kernel for the ball}
The weighted Bergman kernel, metric, volume, canonical functions of $\mathbb{B}^n_r$ with respect to the K\"ahler-Einstein weight $\mu_r^m$ are given by
\begin{align*}
    K_{\mathbb{B}^n_r,\mu^{m}_{r}}(z)&=\frac{(n+1)^{(m-1)n}}{\pi^n r^{2mn}}\frac{(m(n+1)-1)!}{((m-1)(n+1))!}\left(\frac{r^2-|z|^2}{r^2}\right)^{-m(n+1)}\\
    &=\frac{1}{C_m}\frac{r^{2m}}{(r^2-|z|^2)^{m(n+1)}},\\
    G_{\mathbb{B}^n_r,\mu^{m}_{r}}(z)&=m(n+1)\left[\frac{\delta_{ij}}{r^2-|z|^2}+\frac{\overline{z_i}z_j}{(r^2-|z|^2)^2}\right]\\
    \det G_{\mathbb{B}^n_r,\mu^{m}_{r}}(z)&=(m(n+1))^n\frac{r^2}{(r^2-|z|^2)^{n+1}},\\
J_{\mathbb{B}^n_r,\mu^{m}_{r}}(z)&=C_m(m(n+1))^n\frac{(r^2-|z|^2)^{(m-1)(n+1)}}{r^{2(m-1)}},
    \end{align*}
where $C_m:=\frac{\pi^n}{(n+1)^{(m-1)n}}\frac{((m-1)(n+1))!}{(m(n+1)-1)!}$ ($C_1=\frac{\pi^n}{n!}$); cf.\
\cite[Corollary~2]{krantz1996bergman}.
\end{prop}

\begin{rmk}
    In \cite{forelli1974projections}, Forelli and Rudin computed the Bergman kernel function of the unit ball $\mathbb{B}^n$ in $\mathbb{C}^n$ with the admissible weight $(1-|z|^2)^{m}$.
\end{rmk}

\begin{cor}
  The normalized weighted Bergman kernel, metric, volume, canonical functions of $\mathbb{B}^n_r$ are given by
\begin{align*}
    K_{\mathbb{B}^n_r,\mu^{m}_{r}}^{\frac{1}{m}}(z)
    &=\frac{1}{C_m^{1/m}}\frac{r^{2}}{(r^2-|z|^2)^{n+1}},\\
    \tilde{G}_{\mathbb{B}^n_r,\mu^{m}_{r}}(z)&=(n+1)\left[\frac{\delta_{ij}}{r^2-|z|^2}+\frac{\overline{z_i}z_j}{(r^2-|z|^2)^2}\right]\\
    \det \tilde{G}_{\mathbb{B}^n_r,\mu^{m}_{r}}(z)&=(n+1)^n\frac{r^2}{(r^2-|z|^2)^{n+1}},\\
    \tilde{J}_{\mathbb{B}^n_r,\mu^{m}_{r}}(z)&\equiv C_m^{1/m}(n+1)^n.
    \end{align*}  
\end{cor}

\begin{rmk}
   The computation shows that $\lim\limits_{m\rightarrow\infty}C_m^{1/m}=\frac{1}{(n+1)^n}$ and
$
\lim\limits_{m\rightarrow\infty}\tilde{J}_{\mathbb{B}^n_r,\mu^{m}_{r}}(z)\equiv1.
$
\end{rmk}

Proposition \ref{prop: bergman kernel for the ball} also yields the following

\begin{prop}\label{prop: tensors of the unit ball}
At the origin, the weighted Bergman metric and the bisectional curvatures are
\begin{align*}
  \langle X,Y\rangle_{g_{\mathbb{B}^n_r,\mu^{m}_{r}}(0)}&=\frac{m(n+1)}{r^2}\langle X,Y\rangle_{\mathbb{C}^n}, \\
  S_{\mathbb{B}^n_r,\mu^{m}_{r}}(0;X,Y)&=1-\frac{|\langle X,Y\rangle_{\mathbb{C}^n}|^2}{|X|_{\mathbb{C}^n}^2|Y|_{\mathbb{C}^n}^2}\\
  T_{\mathbb{B}^n_r,\mu_{r}^m}(0;X,Y)
    &=\left(1+\frac{1}{m(n+1)}\right)\left(1+\frac{|\langle X,Y\rangle_{\mathbb{C}^n}|^2}{|X|_{\mathbb{C}^n}^2|Y|_{\mathbb{C}^n}^2}\right)\\
  B_{\mathbb{B}^n_r,\mu^{m}_{r}}(0;X,Y)&=\frac{-1}{m(n+1)}\left(1+\frac{|\langle X,Y\rangle_{\mathbb{C}^n}|^2}{|X|_{\mathbb{C}^n}^2|Y|_{\mathbb{C}^n}^2}\right),
\end{align*}
where $\langle X,Y\rangle_{\mathbb{C}^n}:=\sum_{j=1}^nX_j\overline{Y_j}$ the standard hermitian inner product.    
\end{prop}

\begin{lem}\label{lem: minimum integrals of ball}
The minimum integrals of $\mathbb{B}^n_r$ at the origin are given by
\begin{align*}
    I^0_{\mathbb{B}^n_r,\mu^{m}_{r}}(0)&=C_mr^{2mn}\\
    I^1_{\mathbb{B}^n_r,\mu^{m}_{r}}(0;X)&=C_m\frac{r^{2(mn+1)}}{m(n+1)}\frac{1}{|X|_{\mathbb{C}^n}^2},\\
    I^1_{\mathbb{B}^n_r,\mu^{m}_{r}}(0;X|Y)&=C_m\frac{r^{2(mn+1)}}{m(n+1)}\frac{1}{|X|_{\mathbb{C}^n}^2\left(1-\frac{|\langle X,Y\rangle_{\mathbb{C}^n}|^2}{|X|_{\mathbb{C}^n}^2|Y|_{\mathbb{C}^n}^2}\right)},\\
    I^1_{\mathbb{B}^n_r,\mu^{m}_{r}}(0; \partial_k|_{<k})&=C_m\frac{r^{2(mn+1)}}{m(n+1)},\\
    I^2_{\mathbb{B}^n_r,\mu^{m}_{r}}(0;X,Y)&=C_m\frac{r^{2(mn+2)}}{m(n+1)(m(n+1)+1)}\frac{1}{|X|_{\mathbb{C}^n}^2|Y|_{\mathbb{C}^n}^2\left(1+\frac{|\langle X,Y\rangle_{\mathbb{C}^n}|^2}{|X|_{\mathbb{C}^n}^2|Y|_{\mathbb{C}^n}^2}\right)}.
\end{align*}
The formula for $I_{B_r^n,\mu_r^m}^{1}(0;X\mid Y)$
is stated for linearly independent $X,Y$.
The remaining formulas hold for arbitrary nonzero directions.
\end{lem}
\begin{proof}
Note that every holomorphic function has the Taylor expansion on $\mathbb{B}^n_r$.
Since $\mathbb{B}^n_r$ is a Reinhardt domain and the weight $\mu^{m}_{r}$ is radially symmetric, the normalized monomials consist of an orthonormal basis of $\mathcal{A}^2(\mathbb{B}^n_r,\mu^{m}_{r})$. 
One can also use Proposition \ref{prop: tensors of the unit ball}.
\end{proof}

\begin{lem}\label{lem: tensor squeezing}
Let $\Omega$ be a bounded domain in $\mathbb{C}^n$ with an admissible weight  $d\mu_{\Omega}^m:=e^{-(m-1)\varphi_{\Omega}}d\lambda$.
    Suppose that $\mathbb{B}^n_r\subset\Omega\subset\mathbb{B}^n_R$ and $\varphi_R\leq\varphi_{\Omega}\leq\varphi_r$. Then, we have
    $$
    \left(\frac{r}{R}\right)^{2mn}|X|_{\mathbb{B}^n_R,\mu^{m}_{R}(0)}^2
    \leq
    |X|^2_{g_{\Omega},\mu_{\Omega}^m(0)}
    \leq
    \left(\frac{R}{r}\right)^{2mn}|X|_{\mathbb{B}^n_r,\mu^{m}_{r}(0)}^2
    $$
    \[
\left(\frac{r}{R}\right)^{2(mn+1)}
S_{B^n,\mu_{B^n}^m}(0;X,Y)
\le S_{\Omega,\mu_\Omega^m}(0;X,Y)
\\
\le
\left(\frac{R}{r}\right)^{2(mn+1)}
S_{B^n,\mu_{B^n}^m}(0;X,Y)
\]
    $$
    \left(\frac{r}{R}\right)^{4(mn+1)}
    \leq
    \frac{T_{\Omega,\mu_{\Omega}^m}(0;X,Y)}{T_{\mathbb{B}^n,\mu_{\mathbb{B}^n}^m}(0;X,Y)}
    \leq
    \left(\frac{R}{r}\right)^{4(mn+1)}
    $$
    $$
    C_m(m(n+1))^n\frac{r^{2mn(n+1)}}{R^{2n(mn+1)}}
    \leq
    J_{\Omega,\mu_{\Omega}^m}(0)
    \leq
    C_m(m(n+1))^n\frac{R^{2mn(n+1)}}{r^{2n(mn+1)}}
    $$
    $$
    C_m^{\frac{1}{m}}(n+1)^n\frac{r^{2mn(n+\frac{1}{m})}}{R^{2n(mn+1)}}
    \leq
    \tilde{J}_{\Omega,\mu_{\Omega}^m}(0)
    \leq
    C_m^{\frac{1}{m}}(n+1)^n\frac{R^{2mn(n+\frac{1}{m})}}{r^{2n(mn+1)}}
    $$
    
\end{lem}
\begin{proof}
By Theorems 4.3 and 4.4, for all nonzero $X,Y$,
\[
\begin{aligned}
|X|_{g_{\Omega,\mu_\Omega^m}(0)}^2
&=
\frac{I_{\Omega,\mu_\Omega^m}^{0}(0)}
     {I_{\Omega,\mu_\Omega^m}^{1}(0;X)},
\\
T_{\Omega,\mu_\Omega^m}(0;X,Y)
&=
\frac{
I_{\Omega,\mu_\Omega^m}^{1}(0;X)
I_{\Omega,\mu_\Omega^m}^{1}(0;Y)}
{
I_{\Omega,\mu_\Omega^m}^{2}(0;X,Y)
I_{\Omega,\mu_\Omega^m}^{0}(0)}.
\end{aligned}
\]
Moreover, the Bergman canonical function and its normalization satisfy
\[
\begin{aligned}
J_{\Omega,\mu_\Omega^m}(0)
&=
\frac{
\bigl(I_{\Omega,\mu_\Omega^m}^{0}(0)\bigr)^{n+1}}
{
\displaystyle\prod_{k=1}^{n}
I_{\Omega,\mu_\Omega^m}^{1}(0;\partial_k\mid {<k})},
\\
\widetilde J_{\Omega,\mu_\Omega^m}(0)
&=
\frac{
\bigl(I_{\Omega,\mu_\Omega^m}^{0}(0)\bigr)^{n+1/m}}
{
\displaystyle m^n\prod_{k=1}^{n}
I_{\Omega,\mu_\Omega^m}^{1}(0;\partial_k\mid {<k})}.
\end{aligned}
\]
Applying Lemmas 4.2 and 5.6 to these identities gives the stated
estimates for the metric, $T$, $J$, and $\widetilde J$.

It remains to prove the estimate for $S$.
Suppose first that $X,Y$ are linearly independent.
By Theorem 4.4 and Lemma 4.2,
\[
\frac{I_{B_r^n,\mu_r^m}^{1}(0;X)}
     {I_{B_R^n,\mu_R^m}^{1}(0;X\mid Y)}
\le
S_{\Omega,\mu_\Omega^m}(0;X,Y)
\le
\frac{I_{B_R^n,\mu_R^m}^{1}(0;X)}
     {I_{B_r^n,\mu_r^m}^{1}(0;X\mid Y)}.
\]
Substituting the formulas in Lemma 5.6 gives the stated estimate.
If $X,Y$ are linearly dependent, then
\[
S_{\Omega,\mu_\Omega^m}(0;X,Y)
=
S_{B^n,\mu_{B^n}^m}(0;X,Y)
=0,
\]
so the same inequalities hold.
\end{proof}


\section{Asymptotic behavior of the bisectional curvatures of the Bergman metrics with the K\"ahler-Einstein weights}

In this section, we prove the main theorems.

\begin{thm}\label{thm: quantative estimates}
Let $\Omega$ be a bounded pseudoconvex domain and let $p\in\Omega$. Then, the bisectional curvature of the weighted Bergman metric $g^{\rm KE}_{\Omega,m}$ satisfies the following inequality:
\begin{align*}
    &2D_m\left(1-s_{\Omega}(p)^{-4(mn+1)}\right)\\
    \leq&\
    B_{\Omega,m}^{KE}(p;X,Y)+\frac{1}{m(n+1)}\left(1+\frac{|\langle X,Y\rangle_{ g^{\rm KE}_{\Omega,m}(p)}|^2}{|X|_{g^{\rm KE}_{\Omega,m}(p)}^2|Y|_{g^{\rm KE}_{\Omega,m}(p)}^2}\right)\\
    \leq&\
    2D_m\left(1-s_{\Omega}(p)^{4(mn+1)}\right),
\end{align*}
where $D_m:=\frac{m(n+1)+1}{m(n+1)}$.
Moreover, the holomorphic sectional curvature satisfies that
$$
2D_m\left(1-s_{\Omega}(p)^{-4(mn+1)}\right)
    \leq
    H_{\Omega,m}^{KE}(p;X)+\frac{2}{m(n+1)}
    \leq
    2D_m\left(1-s_{\Omega}(p)^{4(mn+1)}\right).
$$
\end{thm}
\begin{proof}
Since the $\mathcal{M}^{\rm KE}_m$-Bergman metric $g^{\rm KE}_{\Omega,m}:=g_{\Omega,\mu^{\rm KE}_{\Omega,m}}$ is invariant under biholomorphisms, we have
\begin{align*}
&B_{\Omega,m}^{KE}(p;X,Y)+\frac{1}{m(n+1)}\left(1+\frac{|\langle X,Y\rangle_{ g^{\rm KE}_{\Omega,m}(p)}|^2}{|X|_{g^{\rm KE}_{\Omega,m}(p)}^2|Y|_{g^{\rm KE}_{\Omega,m}(p)}^2}\right)\\
=&
B_{\Omega,m}^{KE}(p;X,Y)+\frac{1}{m(n+1)}\left(2-S_{\Omega,m}^{KE}(p;X,Y)\right)\\
=&
B_{f_p(\Omega),m}^{KE}(0;df_p(X),df_p(Y))+\frac{1}{m(n+1)}\left(2-S_{f_p(\Omega),m}^{KE}(0;df_p(X),df_p(Y))\right)\\
=&
\frac{m(n+1)+1}{m(n+1)}\left(2-S_{f_p(\Omega),m}^{KE}(0;df_p(X),df_p(Y))\right)-T_{f_p(\Omega),m}^{KE}(0;df_p(X),df_p(Y)).
\end{align*}
Since $\mathbb{B}^n_{s_{\Omega}(p)}\subset f_p(\Omega)\subset \mathbb{B}^n_1$,
the Schwarz lemma for the volume forms \cite{cheng1980existence,mok1983completeness} of the K\"ahler-Einstein metrics with Ricci curvature $-1$ implies that
$$
\det\left(g^{\rm KE}_{\mathbb{B}^n_1}\right)
\leq
\det\left(g^{\rm KE}_{f_p(\Omega)}\right)
\leq
\det\left(g^{\rm KE}_{\mathbb{B}^n_{s_{\Omega}(p)}}\right).
$$
Applying Lemma \ref{lem: tensor squeezing}, we obtain
\[
\begin{aligned}
&s_\Omega(p)^{2(mn+1)}
S_{B^n,m}^{\mathrm{KE}}
\bigl(0;df_p(X),df_p(Y)\bigr)
\\
&\qquad\le
S_{f_p(\Omega),m}^{\mathrm{KE}}
\bigl(0;df_p(X),df_p(Y)\bigr)
\\
&\qquad\le
s_\Omega(p)^{-2(mn+1)}
S_{B^n,m}^{\mathrm{KE}}
\bigl(0;df_p(X),df_p(Y)\bigr),
\end{aligned}
\]
    $$
    \left(s_{\Omega}(p)\right)^{4(mn+1)}
    \leq
    \frac{T_{f_p(\Omega),m}^{KE}(0;df_p(X),df_p(Y))}{T_{\mathbb{B}^n,m}^{KE}(0;df_p(X),df_p(Y))}
    \leq
    \left(\frac{1}{s_{\Omega}(p)}\right)^{4(mn+1)}.
    $$
Then, the conclusion follows from Proposition \ref{prop: tensors of the unit ball} and the Cauchy-Schwarz inequality:
$$
0\leq\frac{|\langle X,Y\rangle_{\mathbb{C}^n}|^2}{|X|_{\mathbb{C}^n}^2|Y|_{\mathbb{C}^n}^2}\leq1.
$$
\end{proof}

As a generalization of the classical results on the holomorphic sectional curvature of the Bergman metric for strongly pseudoconvex domains \cite{klembeck1978kahler,kim1996boundary,zhang2017intrinsic}, the following theorem is obtained as a corollary of Theorem \ref{thm: quantative estimates}. It extends the analysis to the ``bisectional" curvature of the ``weighted" Bergman metric $g^{\rm KE}_{\Omega,m}$ providing a more general geometric description.

\begin{thm}\label{thm: asymptotic behavior of the curvatures}
    Let $\Omega$ be a bounded pseudoconvex domain and let $\{p_j\} \subset \Omega$ be a sequence such that $\lim\limits_{j\rightarrow\infty}s_{\Omega}(p_j)=1$. Then, for each fixed integer $m\ge1$, we have
    $$
    \lim\limits_{j\rightarrow\infty}\left(B_{\Omega,m}^{KE}(p_j;X,Y)+\frac{1}{m(n+1)}\left(1+\frac{|\langle X,Y\rangle_{ g^{\rm KE}_{\Omega,m}(p_j)}|^2}{|X|_{g^{\rm KE}_{\Omega,m}(p_j)}^2|Y|_{g^{\rm KE}_{\Omega,m}(p_j)}^2}\right)\right)=0
    $$
    uniformly over all nonzero $X,Y\in\mathbb C^n$.
\end{thm}

\begin{proof}
By Theorem 6.1, for each fixed $m\ge1$,
\[
\begin{aligned}
&\sup_{X,Y\ne0}
\left|
B_{\Omega,m}^{\mathrm{KE}}(p_j;X,Y)
+
\frac{2-S_{\Omega,m}^{\mathrm{KE}}(p_j;X,Y)}
     {m(n+1)}
\right|
\\
&\qquad\le
2D_m\bigl(s_\Omega(p_j)^{-4(mn+1)}-1\bigr)
\longrightarrow0
\qquad (j\to\infty).
\end{aligned}
\]
\end{proof}

\begin{rmk}
    Note that by the biholomorphic invariance of the weighted Bergman metric $g^{\rm KE}_{\Omega,m}$, we have
$$
\frac{|\langle X,Y\rangle_{ g^{\rm KE}_{\Omega,m}(p)}|^2}{|X|_{g^{\rm KE}_{\Omega,m}(p)}^2|Y|_{g^{\rm KE}_{\Omega,m}(p)}^2}
=
\frac{|\langle df_{p}(X),df_{p}(Y)\rangle_{ g^{\rm KE}_{f_{p}(\Omega),m}(0)}|^2}{|df_{p}(X)|_{g^{\rm KE}_{f_{p}(\Omega),m}(0)}^2|df_{p}(Y)|_{g^{\rm KE}_{f_{p}(\Omega),m}(0)}^2}.
$$
On the other hand, the bisectional curvature of the unit ball is given by
\begin{align*}
B_{g_{\mathbb{B}^n,m}^{KE}}(p;X,Y)
&=
B_{g_{\mathbb{B}^n,m}^{KE}}(0;d\hat{f}_p(X),d\hat{f}_p(Y))\\
&=\frac{-1}{m(n+1)}\left(1+\frac{|\langle d\hat{f}_p(X),d\hat{f}_p(Y)\rangle_{\mathbb{C}^n}|^2}{|d\hat{f}_p(X)|_{\mathbb{C}^n}^2|d\hat{f}_p(Y)|_{\mathbb{C}^n}^2}\right)\\
&=\frac{-1}{m(n+1)}
\left(
1+
\frac{|\langle d\hat f_p(X),d\hat f_p(Y)\rangle_
{g_{\mathbb B^n,m}^{\rm KE}(0)}|^2}
{|d\hat f_p(X)|_{g_{\mathbb B^n,m}^{\rm KE}(0)}^2
 |d\hat f_p(Y)|_{g_{\mathbb B^n,m}^{\rm KE}(0)}^2}
\right),
\end{align*}
where $\hat f_p$ is an automorphism of the unit ball sending $p$ to $0$.
Therefore, Theorem \ref{thm: asymptotic behavior of the curvatures}
shows that the bisectional curvature of $g_{\Omega,m}^{\rm KE}$
asymptotically satisfies the curvature formula of the unit ball
as $s_\Omega(p)\to1$.
In particular, this holds near the boundary of a smoothly bounded
strongly pseudoconvex domain.
\end{rmk}

\begin{prop}\label{prop:comparison}
Let $\Omega\subset\mathbb C^n$ be a bounded pseudoconvex domain.
For each $0<r<1$ there exist $C=C(n,r)>0$ and $m_0=m_0(n,r)$ such that,
for every $p$ with $s_\Omega(p)\ge r$, every $m\ge m_0$, and all
nonzero $X,Y\in T^{1,0}_p\Omega$,
\[
 \begin{aligned}
 |S^{\KE}_{\Omega,m}(p;X,Y)-S_{g^{\KE}_\Omega}(p;X,Y)|
 &\le \frac{C}{m^2},\\
 |\widetilde B^{\KE}_{\Omega,m}(p;X,Y)-B^{\KE}_\Omega(p;X,Y)|
 &\le \frac{C}{m}.
 \end{aligned}
\]
\end{prop}

The proof follows the argument of \cite[Theorem~5.5]{yoo2025invariant} in the author's earlier paper. In that proof, Yeung's estimates for uniformly squeezing domains, recalled in \cite[Theorem~5.3]{yoo2025invariant}, provide the local bounds needed
to apply \cite[Corollary~4.11]{yoo2025invariant}.
Here we assume only that $s_\Omega(p)\ge r$ at the points under consideration. 
We obtain the required uniform local bounds from Yau's Schwarz lemmas and local regularity for the K\"ahler--Einstein equation.
The estimates in \cite[Section~4]{yoo2025invariant} then apply uniformly in $p$ and the direction, and polarization gives the conclusion.

\begin{proof}
Fix $p$ with $s_\Omega(p)\ge r$. By
\cite[Theorem~2.1]{deng2012some}, there exists a holomorphic
embedding $f_p:\Omega\to\B^n_1$ such that
\[
 f_p(p)=0,\qquad
 \B^n_r\subseteq\B^n_{s_\Omega(p)}
 \subseteq f_p(\Omega)\subseteq\B^n_1.
\]
Yau's metric and volume Schwarz lemmas
\cite[Theorems~2 and~3]{Ya}, applied to the above inclusions, yield
\[
 c_n g_{\mathrm{Euc}}\le g^{\KE}_{f_p(\Omega)}
 \le C(n,r)g_{\mathrm{Euc}}
 \quad\text{on }\B^n_{r/2}.
\]
Here the metric Schwarz lemma gives the lower bound, while the volume Schwarz lemma, together with this lower bound, gives the upper bound; cf.\ the proof of
\cite[Proposition~3]{yeung2009}.

We fix this $c_n$ throughout. All constants denoted by $C$ or $C(n,r)$ may vary from line to line and depend only on $n$ and $r$.
For K\"ahler--Einstein metrics, higher-order interior bounds follow from the complex Evans--Krylov estimate
(see \cite[Theorem~5.15]{Bl}) and interior Schauder estimates.
For a convenient quantitative estimate, we use the parabolic result \cite[Corollary~1.2]{SW}. 
Since
$\Ric(g^{\KE}_{f_p(\Omega)})=-g^{\KE}_{f_p(\Omega)}$, the family
$t\mapsto(1+t)g^{\KE}_{f_p(\Omega)}$ solves the unnormalized K\"ahler--Ricci flow. 
On a fixed bounded time interval, chosen
independently of $p$, the preceding comparison with
$g_{\mathrm{Euc}}$ remains uniform in $p$.
Applying that corollary with the Euclidean background at a fixed positive time gives, for every integer $d\ge0$,
\[
 \|g^{\KE}_{f_p(\Omega)}\|_{C^d(\B^n_{r/4})}
 \le C_d(n,r).
\]
Here the $C^d$-norms are taken in the Euclidean coordinates, coefficientwise for the metric, and include all $z$- and $\bar z$-derivatives of total order at most $d$.
By the definition
$\varphi^{\KE}_{f_p(\Omega)}
=\log\det g^{\KE}_{f_p(\Omega)}$
and the eigenvalue bounds above, it follows that
\begin{equation}
 \|\varphi^{\KE}_{f_p(\Omega)}\|_{C^d(\B^n_{r/4})}
 \le C_d(n,r).
 \label{localgeometry}
\end{equation}

We now use \eqref{localgeometry} with $d=6$ and write
$\varphi=\varphi^{\KE}_{f_p(\Omega)}$.
Apply \cite[Lemma~4.4]{yoo2025invariant} at
$0\in f_p(\Omega)$: let $w=f(z)$ be the local holomorphic coordinate change, with $f(0)=0$, and let $h$ be the holomorphic polynomial constructed there. The normalized potential is
\[
 \Phi(w):=(\varphi-2\operatorname{Re}h)\circ f^{-1}(w).
\]
By \eqref{localgeometry} and the metric lower bound above, the polynomial coefficients of $f$ and $h$, as well as $df(0)^{-1}$, are uniformly bounded.
Thus we may restrict $f$ to a biholomorphism $f:W\to\B^n_{r_0}$, with $W\subset\B^n_{r/4}$ and
$r_0=r_0(n,r)>0$, such that
\[
 \|\Phi\|_{C^6(\B^n_{r_0})}\le C(n,r).
\]
These bounds remain uniform after a unitary change of the $w$-coordinates.

In \cite[proof of Proposition~4.6]{yoo2025invariant},
the metric lower bound above gives $C''=c_n$, while the fixed cutoff and the bound on the coordinate differential give $C'_\varphi(0)\le C(n,r)$.
Take $j'=3$ in that proof, so that the same cutoff applies to all three minimum integrals, indexed by $j=0,1,2$.
There exists $m_0=m_0(n,r)$ such that
\[
 \frac{\log m}{\sqrt m}<r_0,\qquad
 c_n-\frac{(n+3)C(n,r)}{(\log m)^2}\ge\frac{c_n}{2}
 \qquad (m\ge m_0).
\]

The uniform local bounds above make the $O(m^{-2})$ term
inside the parentheses in \cite[(4.6)]{yoo2025invariant}
uniform in $p$ and the direction.
They also make the constant in \cite[(4.8)]{yoo2025invariant} uniform when that estimate is applied to the difference between the test function and the minimizer. 
Thus \cite[Lemma~4.9]{yoo2025invariant} gives the corresponding minimum integral expansions for the weight $e^{-m\varphi}$, with uniform $O(m^{-2})$ terms inside the parentheses, for all $m\ge m_0$, after increasing this common $m_0$ if necessary; cf.\ \cite[Remark~4.13]{yoo2025invariant}.

For $m\ge m_0+1$, we apply
\cite[Corollary~4.11]{yoo2025invariant} to
$\mu^{\KE}_{f_p(\Omega),m}=e^{-(m-1)\varphi}$.
Using $\Ric(g^{\KE}_{f_p(\Omega)})=-g^{\KE}_{f_p(\Omega)}$
in the metric and curvature expansions of that corollary, we obtain, after normalization and pullback by $f_p$, the following estimates for $m\ge m_0+1$ and
$0\ne X\in T^{1,0}_p\Omega$:
\begin{equation}
 \begin{gathered}
 \frac{|X|_{\widetilde g^{\KE}_{\Omega,m}(p)}^2}
 {|X|_{g^{\KE}_\Omega(p)}^2}
 =1+O(m^{-2}),\qquad
 H_{\widetilde g^{\KE}_{\Omega,m}}(p;X)
 -H_{g^{\KE}_\Omega}(p;X)
 =O(m^{-1}).
 \end{gathered}
 \label{direction}
\end{equation}
Here the constants implicit in the $O$-terms are bounded by $C(n,r)$, uniformly in $p$ and $X$.

Enlarge $m_0$ if necessary so that the squared-length ratio in \eqref{direction} lies in $[1/2,2]$ for $m\ge m_0$.
For any K\"ahler metric $g$, Hermitian polarization and the K\"ahler symmetries of the curvature tensor give, respectively,
\[
 \begin{aligned}
 4\langle X,Y\rangle_g
 &=\sum_{k=0}^3 \sqrt{-1}^k|X+\sqrt{-1}^kY|_g^2,\\
 4|X|_g^2|Y|_g^2 B_g(X,Y)
 &=\tfrac14\sum_{k=0}^3
   |X+\sqrt{-1}^kY|_g^4H_g(X+\sqrt{-1}^kY)
   -|X|_g^4 H_g(X)-|Y|_g^4 H_g(Y).
 \end{aligned}
\]
To estimate $S$, apply the first identity and \eqref{direction}
to arbitrary $X,Y\ne0$. 
This gives
\[
 \frac{
 \big|\langle X,Y\rangle_{\widetilde g^{\KE}_{\Omega,m}}
 -\langle X,Y\rangle_{g^{\KE}_\Omega}\big|}
 {|X|_{g^{\KE}_\Omega}|Y|_{g^{\KE}_\Omega}}
 \le \frac{C}{4m^2}\sum_{k=0}^3
 \left|
 \frac{X}{|X|_{g^{\KE}_\Omega}}
 +\sqrt{-1}^k\frac{Y}{|Y|_{g^{\KE}_\Omega}}
 \right|_{g^{\KE}_\Omega}^2
 =\frac{2C}{m^2}.
\]
Since $S$ is invariant under constant rescaling of the metric, its definition, Cauchy--Schwarz, the preceding bound and \eqref{direction} give
\[
 |S^{\KE}_{\Omega,m}-S_{g^{\KE}_\Omega}|
 \le 2\left|
 \frac{\langle X,Y\rangle_{\widetilde g^{\KE}_{\Omega,m}}}
 {|X|_{\widetilde g^{\KE}_{\Omega,m}}
  |Y|_{\widetilde g^{\KE}_{\Omega,m}}}
 -
 \frac{\langle X,Y\rangle_{g^{\KE}_\Omega}}
 {|X|_{g^{\KE}_\Omega}|Y|_{g^{\KE}_\Omega}}
 \right|
 \le Cm^{-2}.
\]

For the curvature estimate, \eqref{localgeometry} and the metric lower bound above give
$|H_{g^{\KE}_\Omega}|,|B^{\KE}_\Omega|\le C(n,r)$ at $p$.
By \eqref{direction},
\[
 \frac{|X|_{\widetilde g^{\KE}_{\Omega,m}}^4}
 {|X|_{g^{\KE}_\Omega}^4}
 H_{\widetilde g^{\KE}_{\Omega,m}}(X)
 -H_{g^{\KE}_\Omega}(X)
 =
 (1+O(m^{-2}))
 \big(H_{g^{\KE}_\Omega}(X)+O(m^{-1})\big)
 -H_{g^{\KE}_\Omega}(X)
 =O(m^{-1}).
\]
Applying the second polarization identity to the two metrics and subtracting gives
\[
 \begin{aligned}
 &\left|
 \frac{|X|_{\widetilde g^{\KE}_{\Omega,m}}^2
       |Y|_{\widetilde g^{\KE}_{\Omega,m}}^2}
 {|X|_{g^{\KE}_\Omega}^2|Y|_{g^{\KE}_\Omega}^2}\,
 \widetilde B^{\KE}_{\Omega,m}(p;X,Y)
 -B^{\KE}_\Omega(p;X,Y)
 \right|\\
 &\qquad\le\frac{C}{4m}
 \left(
 \frac14\sum_{k=0}^3
 \left|
 \frac{X}{|X|_{g^{\KE}_\Omega}}
 +\sqrt{-1}^k\frac{Y}{|Y|_{g^{\KE}_\Omega}}
 \right|_{g^{\KE}_\Omega}^4
 +1+1
 \right)
 \le\frac{18C}{4m}.
 \end{aligned}
\]
Using \eqref{direction} and the bound for $B^{\KE}_\Omega$, we conclude that
\[
 \widetilde B^{\KE}_{\Omega,m}(p;X,Y)
 =
 \frac{B^{\KE}_\Omega(p;X,Y)+O(m^{-1})}
 {1+O(m^{-2})}
 =
 B^{\KE}_\Omega(p;X,Y)+O(m^{-1}).\qedhere
\]
\end{proof}

As a corollary of Theorem \ref{thm: asymptotic behavior of the curvatures} and Proposition \ref{prop:comparison}, we provide an alternative proof, using weighted Bergman metrics, of the following curvature asymptotics of the K\"ahler-Einstein
metric $g_{\Omega}^{\rm KE}$.
The corresponding boundary asymptotics were established by Cheng and Yau {\cite[Proposition 1.5]{cheng1980existence}} for smoothly bounded strongly pseudoconvex domains, and extended by Gontard {\cite[Theorem 1.4]{gontard2019kahler}} to boundary points at which the squeezing function tends to one.

\begin{cor}[{\cite[Theorem 1.4]{gontard2019kahler}}]\label{cor: ke asymptotic}
Let $\Omega\subset\mathbb C^n$ be a bounded pseudoconvex domain. 
If $(p_j)\subset\Omega$ satisfies $s_\Omega(p_j)\to1$, then
\[
 \lim_{j\to\infty}\sup_{X,Y\ne0}
 \left|B^{\KE}_\Omega(p_j;X,Y)+\frac{1}{n+1}
 \left(1+\frac{|\langle X,Y\rangle_{g^{\KE}_\Omega}|^2}
 {|X|_{g^{\KE}_\Omega}^2|Y|_{g^{\KE}_\Omega}^2}\right)\right|=0,
 \qquad X,Y\in T^{1,0}_{p_j}\Omega .
\]
\end{cor}

\begin{proof}
By the definition of $S$ in Theorem~4.4, the factor in parentheses in the statement is
\[
 1+\frac{|\langle X,Y\rangle_{g^{\KE}_\Omega}|^2}
 {|X|_{g^{\KE}_\Omega}^2|Y|_{g^{\KE}_\Omega}^2}
 =2-S_{g^{\KE}_\Omega}(p_j;X,Y).
\]
For $g^{\KE}_{\Omega,m}$ the same identity gives
$2-S^{\KE}_{\Omega,m}(p_j;X,Y)$ in Theorem~6.2.
Choose $J$ so that $s_\Omega(p_j)\ge1/2$ for $j\ge J$. In the suprema below, all quantities are evaluated at $p_j$ and $X,Y\in T^{1,0}_{p_j}\Omega$.
For $m\ge m_0(n,1/2)$, Proposition~\ref{prop:comparison} gives
\[
 \sup_{j\ge J}\left|
 \sup_{X,Y\ne0}\left|\widetilde B^{\KE}_{\Omega,m}
                   +\frac{2-S^{\KE}_{\Omega,m}}{n+1}\right|
 -
 \sup_{X,Y\ne0}\left|B^{\KE}_\Omega+\frac{2-S_{g^{\KE}_\Omega}}{n+1}\right|
 \right|\le \frac C m+\frac{C}{(n+1)m^2}.
\]
Thus the first supremum over $X,Y$ converges to the second uniformly in $j\ge J$ as $m\to\infty$. 
For each fixed $m$, its limit as $j\to\infty$ is zero by Theorem~6.2, since $\widetilde B^{\KE}_{\Omega,m}=mB^{\KE}_{\Omega,m}$
and the bounds of Theorem~6.1 are uniform in $X,Y$. Hence the limits can be interchanged, giving
\[
 \begin{aligned} \lim_{j\to\infty}\sup_{X,Y\ne0}\left|B^{\KE}_\Omega+\frac{2-S_{g^{\KE}_\Omega}}{n+1}\right|
 &=\lim_{j\to\infty}\lim_{m\to\infty}\sup_{X,Y\ne0}
 \left|\widetilde B^{\KE}_{\Omega,m}+\frac{2-S^{\KE}_{\Omega,m}}{n+1}\right|\\
 &=\lim_{m\to\infty}\lim_{j\to\infty}\sup_{X,Y\ne0}
 \left|\widetilde B^{\KE}_{\Omega,m}+\frac{2-S^{\KE}_{\Omega,m}}{n+1}\right|=0 .
 \end{aligned}
\]
\end{proof}

\textbf{Acknowledgements}.
The author dedicates this paper to Professor Kang-Tae Kim in honor of his outstanding contributions to several complex variables and complex geometry.
The author is grateful to the anonymous referee for a careful reading of the manuscript and for valuable comments and suggestions that led to substantial improvements in the proofs and presentation.
This work was supported by Incheon National University Research Grant in 2023.

\bibliographystyle{abbrev}
\bibliography{reference}

\end{document}